\documentclass[11pt]{amsart}
\usepackage{amssymb, amsthm, amsfonts}
\usepackage[all]{xy}

\newcommand{\cO}{\mathcal{O}}
\newcommand{\cX}{\mathcal{X}}
\newcommand{\cY}{\mathcal{Y}}
\newcommand{\cI}{\mathcal{I}}

\newcommand{\cL}{\mathcal{L}}
\newcommand{\cH}{\mathcal{H}}

\renewcommand{\AA}{\mathbb{A}}
\newcommand{\ZZ}{\mathbb{Z}}

\newcommand{\QQ}{\mathbb{Q}}
\newcommand{\FF}{\mathbb{F}}

\newcommand{\newword}[1]{\textbf{\emph{#1}}}

\newcommand{\Spec}{\mathrm{Spec \ }}
\newcommand{\Tr}{\mathrm{Tr}}
\newcommand{\rr}{\mathrm{red}}
\newcommand{\Frac}{\mathrm{Frac}\ }
\newcommand{\Gal}{\mathrm{Gal}}
\newcommand{\Tor}{\mathrm{Tor}}

\newcommand{\fp}{\mathfrak{p}}
\newcommand{\fq}{\mathfrak{q}}
\newcommand{\fr}{\mathfrak{r}}

\newtheorem{Theorem}{Theorem}
\newtheorem{Proposition}[Theorem]{Proposition}
\newtheorem{Lemma}[Theorem]{Lemma}

\theoremstyle{definition}
\newtheorem{Example}[Theorem]{Example}
\newtheorem{Remark}[Theorem]{Remark}

\begin{document}

\title[Pulling back Frobenius split subvarieties]{Frobenius split subvarieties pull back in almost all characteristics}
\author{David E Speyer}
\maketitle

\begin{abstract}
Let $\cX$ and $\cY$ be schemes of finite type over $\Spec \ZZ$ and let $\alpha: \cY \to \cX$ be a finite map. We show the following holds for all sufficiently large primes $p$: If $\phi$ and $\psi$ are any splittings on $\cX \times \Spec \FF_p$ and $\cY \times \Spec \FF_p$, such that the restriction of $\alpha$ is compatible with $\phi$ and $\psi$, and $V$ is any compatibly split subvariety of $(\cX \times \Spec \FF_p, \phi)$, then the reduction $\alpha^{-1}(V)^{\rr}$ is a  compatibly split subvariety of $(\cY \times \Spec \FF_p, \psi)$. This is meant as a tool to aid in listing the compatibly split subvarieties of various classically split varieties.
\end{abstract}

Let $X$ be a scheme over $\FF_p$. 
A \newword{Frobenius splitting} on $X$ is a map $\phi: \cO_X \to \cO_X$ obeying  $\phi(a+b) = \phi(a)+\phi(b)$, $\phi(a^p b) = a \phi(b)$ and $\phi(1)=1$. Here $\cO_X$ is the structure sheaf of $X$ and $a$ and $b$ are sections of $\cO_X$ on any open set of $X$.
A map $\phi$ which obeys only the first two conditions is called a \newword{near splitting}.

\begin{Remark} If $X=\Spec A$ is affine and $\phi: A \to A$ is a map obeying the above conditions, then $\phi$ automatically extends to such a map $\cO_X \to \cO_X$, and uniquely so. The reader will lose very little by thinking of a splitting as a map of rings rather than sheaves. \end{Remark}

We say that a subvariety $Z$ of $X$ is \newword{compatibly split} if the splitting $\phi$ descends to a splitting on $Z$. This is equivalent to requiring that $\phi(\cI_Z) \subseteq \cI_Z$, where $\cI_Z$ is the ideal sheaf of $Z$. It is easy to show (\cite[Proposition 1.2.1]{BK}) that, if $X$ supports a Frobenius splitting, then $\cO_X$ is reduced, which is why we discuss compatibly split subvarieties and don't need to consider more general subschemes.

Let $(X, \phi)$ and $(Y, \psi)$ be two $\FF_p$-schemes with Frobenius splittings, and $\alpha: Y \to X$ a finite map of $\FF_p$ schemes. Suppose that $\alpha$ is \newword{compatible} with the splittings $\phi$ and $\psi$, which means that $\psi(\alpha^{\ast} (a)) = \alpha^{\ast}(\phi(a))$. 
If $W$ is a split subvariety of $Y$, then $\alpha(W)$ is closed in $X$ and it is easy to see that $\alpha(W)$ is compatibly split in $X$. 

We now consider the reverse situation. Suppose that $V$ is a compatibly split subvariety of $X$. The scheme theoretic preimage $\alpha^{-1}(V)$ is often not reduced, and hence cannot be compatibly split. 
However, we can ask whether the reduction $\alpha^{-1}(V)^{\rr}$ is compatibly split.
The answer may still be no.

\begin{Example} \label{wild}
Let $X = \Spec \FF_2[w]$ and let $\phi$ be the unique Frobenius splitting with $\phi(w)=w$. The origin $w=0$ is compatibly split; we denote this point by $q$.
Let $Y = \Spec \FF_2[x,x^{-1}]$ and map $\alpha: Y \to X$ by $\alpha^{\ast}(w) = x+x^{-1}$. The splitting $\phi$ extends uniquely to a splitting on $Y$: We compute $w = \phi(w) = \phi(x)+\phi(x^{-1})=\phi(x)+\phi(x \cdot x^{-2}) = \phi(x)(1+x^{-1})$ so $\phi(x) = w/(1+x^{-1}) = x+1$. The preimage $\alpha^{-1}(q)$ is $x=1$, so we see that $x+1$ vanishes at $\alpha^{-1}(q)$ but $\phi(x+1) = x$ does not. Thus, taking $V = \{ q \}$, this is an example where  $\alpha^{-1}(V)^{\rr}$  is not compatibly split.
\end{Example}
%

This example is wildly ramified, and would not occur for a degree two cover in odd characteristic. 
This leads one to suspect that, if  we have some family of maps with varying characteristic, then, for almost all primes, such a relationship cannot occur. 
Our main theorem makes this statement precise.

\begin{Theorem} \label{MT}
Let $\cX$ and $\cY$ be schemes of finite type over $\Spec \ZZ$ and let $\alpha: \cY \to \cX$ be a finite map. Then there is a positive integer $N$ such that the following holds for any prime $p > N$: If $\phi$ and $\psi$ are any splittings on $\cX \times \Spec \FF_p$ and $\cY \times \Spec \FF_p$, such that the restriction of $\alpha$ is compatible with $\phi$ and $\psi$, and $V$ is any compatibly split subvariety of $(\cX \times \Spec \FF_p, \phi)$, then the reduction $\alpha^{-1}(V)^{\rr}$ is a  compatibly split subvariety of $(\cY \times \Spec \FF_p, \psi)$.
\end{Theorem}

\begin{Remark}
There is no difficulty, other than notation, in extending this result to rings of integers in number fields other than $\QQ$.
\end{Remark}

For the rest of the paper, we will adopt the shorthands $\cX(p)$ and $\cY(p)$ for $\cX \times \Spec \FF_p$ and $\cY \times \Spec \FF_p$.

We will prove Theorem~\ref{MT} in Section~\ref{proofs}; the proof is quite elementary.
Before that, we will explain in Section~\ref{Motivation} why this result is interesting.

I would like to express my thanks to Allen Knutson for introducing me to Frobenius splitting, and to Karl Schwede and Jenna Rajchgot for helpful conversations and for encouraging me to publish this result. I was supported by a Clay Research Fellowship when I worked out most of these results, and by NSF Grant DMS1600223 for some of the period of writing it.

\section{Motivation} \label{Motivation}


Let $X$ be a variety of finite type over $\FF_p$, equipped with a Frobenius splitting. Kumar and Mehta~\cite{KM} and independently Schwede~\cite{S}, showed that there are only finitely many compatibly split subvarieties of $X$.
If $X$ is an affine variety over $\FF_p$, given by explicit equations in $\FF_p[x_1, \ldots, x_n]$, then Katzman and Schwede~\cite{KS} give an algorithm for computing all of the compatibly split ideals. 

Moreover, suppose that $\cX$ is projective over $\Spec \ZZ$ and we have a Frobenius splitting on each $\cX(p)$. The result of Schwede and Tucker~\cite{ST} gives a uniform bound, independent of $p$, for the number of compatibly split subvarieties of $\cX(p)$.
One of the most common ways that we can obtain splittings on all $\cX(p)$ at once is the following: Suppose that $\cX$ is smooth over $\Spec \ZZ$ and let $\sigma$ be a section of $\omega_{\cX/\Spec \ZZ}^{-1}$. Then $\sigma^{p-1}$ induces a near-splitting on $\cX(p)$ (\cite[Theorem 1.3.8]{BK}).
Knutson~\cite{Knutson} has promoted the problem of explicitly listing these compatibly split varieties of $\cX(p)$ for various classical splittings: See~\cite{KLS} for $X=G/P$ and $\sigma$ the divisor which vanishes on the Schubert and opposite Schubert varieties; see~\cite{Rajchgot} for $X = \mathrm{Hilb}^n \AA^2$ and $\sigma$ induced by the standard volume form on $\AA^{2n}$.  
Katzmann and Schwede's method is only of partial use in these cases, because it works with one particular prime and one particular variety. It can compute all compatibly split subvarieties of $\mathrm{Hilb}^n\  \mathrm{Spec}\  \FF_p[x,y]$ for some particular $p$ and $n$, but it does not make sense to ask about all $p$ and/or all $n$.

To attack Knutson's problem, one generally proceeds inductively on the codimension of the compatibly split subvariety. 
If $X$ is R1, it is easy to describe the compatibly split divisors: They are the zeroes of the section $\sigma$.
If $\sigma$ vanishes on a divisor $D$, and $X$ is R1 along $D$, then the splitting on $D$ is induced by $(\mathrm{Res}(\sigma))^{-1}$, where $\mathrm{Res} : \omega_X(D) \to \omega_D$ is the residue map.
One can then attempt inductively to describe the compatibly split subvarieties of $D$.
Unfortunately, it is not always true that a compatibly split subvariety is contained in a compatibly split divisor, but this does hold if $X$ is smooth on the locus where $\sigma$ is nonzero.

At some point in this inductive process, we may encounter subvarieties which are not R1. A natural approach to take is to replace them by their normalizations.
Theorem~\ref{MT} tells us that we may lose some information by doing so, but we only lose information at finitely many primes.

We conclude with an example of finding all compatibly split subvarieties for a given splitting. We do the computation twice, once without using normalization and once with normalization.

\begin{Example}
Let $\cX = \Spec \ZZ[u,v,w]$. Let $\sigma$ be the section
\[ \sigma := \frac{uvw - v^2-w^2}{du \wedge dv \wedge dw} \]
of $\omega_{\cX/\Spec \ZZ}^{-1}$. For every $p$, we have a splitting on $\cX(p)$ induced by $\sigma^{p-1}$. (For example, we can use the criterion of~\cite[Theorem 4]{Knutson} to see that this near-splitting is a splitting.)
We first work out the compatibly split subvarieties of $\cX(p)$ without using normalization.

Since $\sigma$ vanishes on the surface $\cH:= \{ uvw=v^2+w^2 \}$, we see that $\cH(p)$ is compatibly split for every $p$.
The surface $\cH$ is singular on the line $\cL = \{ v=w=0 \}$ and otherwise smooth.
We have $\sigma^{-1}  = \tfrac{d(uvw-v^2-w^2)}{uvw-v^2-w^2} \wedge \tfrac{dv \wedge d w}{vw}$ so the residue of $\sigma^{-1}$ along $\cH$ is $\tfrac{dv \wedge d w}{vw}$. We deduce that, on the smooth locus of $\cH(p)$, the splitting is induced by $\tau^{p-1}$ where $\tau$ is the section $\tfrac{vw}{dv \wedge dw}$ of $\omega^{-1}_{\cH/\Spec \ZZ}$.

Away from $\cL$, the section $\tau$ is nonvanishing, so all compatibly split subvarieties of $\cH(p)$ are contained in $\cL(p)$. The line $\cL(p)$ is compatibly split. For $p=2$, the splitting on $\cL(p)$ is induced by $u/du$. 
For general $p$, the splitting on $\cL(p)$ is induced by $(du)^{1-p} \sum u^k \left( [u^{k} v^{p-1} w^{p-1}] (uvw-v^2-w^2)^{p-1} \right)$, where $[m] f$ is the coefficient of the monomial $m$ in the polynomial $f$. Some manipulation with binomial coefficients simplifies this expresssion: when $p$ is odd, the splitting on $\cL(p)$ is induced by $( u^2-4)^{(p-1)/2} (du)^{-(p-1)}$. 
We note that this is not a $(p-1)$-st power of an anticanonical section, and the computations we had to perform to find it were more difficult than taking residues.

When $p=2$ we further see that the origin $u=v=w=0$ is split since $u/du$ vanishes there. For odd $p$, there are no split points because the zeroes of $(u^2-4)^{(p-1)/2}$ have multiplicity less than $p-1$.

We now repeat the computation using normalizations. As before, since $\sigma$ vanishes on $\cH$, we know that $\cH(p)$ is compatibly split.
Normalizing $\cH$ adds the functions $x:=v/w$ and $x^{-1}$, which obey the monic polynomials $x^2-u x + 1 = x^{-2} - u x^{-1} + 1 =0$. We have $u = x+x^{-1}$ and $v=wx$, so the normalization $\widetilde{\cH}$ is $\Spec \ZZ[w, x^{\pm 1}]$. The splitting on $\widetilde{\cH}$ is induced by $\widetilde{\tau}$, the pullback of $\tau$ to $\widetilde{\cH}$. We compute $\widetilde{\tau} = \tfrac{wx}{dw \wedge dx}$. The section $\widetilde{\tau}$ vanishes only on $\widetilde{\cL} = \{ (w,x) \in \widetilde{\cH} : w=0 \}$ (recall that we have inverted $x$). The splitting on $\widetilde{\cL}(p)$ is induced by $(x/dx)^{p-1}$  (computed by taking residues again). Since we have inverted $x$, the section $x/dx$ is nowhere vanishing and there are no split points. Taking the images of $\widetilde{\cH}$ and $\widetilde{\cL}$ in $\AA^3$, we see that $\cH$ and $\cL$ are compatibly split.

We remark that the map $\widetilde{\cL}(2) \to \cL(2)$ is the wildly ramified cover from Example~\ref{wild}.
\end{Example}

The benefit of the second computation is that we were able to do all our computations by repeatedly taking residues of differential forms, and all computations were uniform in $p$.
The benefit of the first computation is that we found the split point $u=v=w=0$ in $\cX(2)$, which the computation with normalizations missed. The split point has a preimage $(w,x)=(0,1)$ in $\widetilde{\cH}$, but this point is not compatibly split so we did not find it.

One can imagine either set of benefits being advantageous in different situations, but at least sometimes, we will prefer the benefits of the second method.
The purpose of Theorem~\ref{MT} is to assure us that we can only lose information at finitely many primes if we select that method.

\section{Proofs}~\label{proofs}

The letter $p$ will always denote a prime. We begin with some lemmas about field extensions in characteristic $p$.

\begin{Lemma} 
Let $L/K$ be a finite dimensional extension of fields of characteristic $p$. Then $\Tr_{L/K}(x^p) = \Tr_{L/K}(x)^p$.
\end{Lemma}

\begin{proof}
If $L/K$ is not separable, then both sides are zero, so assume that $L/K$ is separable.
Let $e_1$, \dots, $e_d$ be a basis for $L$ over $K$. 
Since $L/K$ is separable,  the extensions $L$ and $K^{1/p}$ are linearly disjoint, so $e_1$, $e_2$, \dots $e_d$ is a basis for $L^{1/p}$ over $K^{1/p}$.
We deduce that $e_1^p$, \dots, $e_d^p$ are a basis for $L$ over $K$.

If $(a_{ij})$ is the matrix of multiplication by $x$ in the basis $e_i$, then $(a_{ij}^p)$ is the matrix for multiplication by $x^p$ in the basis $e_i^p$.
Using the first basis $\sum a_{ii}=\Tr(x)$ and, using the second basis, $\sum a_{ii}^p = \Tr(x^p)$. So $\Tr(x^p) = \Tr(x)^p$.
\end{proof}

\begin{Lemma} \label{Key}
Let $L/K$ be a finite dimensional extension of fields of characteristic $p$ and let $\phi$ be a splitting on $L$ which restricts to a splitting of $K$.
Then $\Tr_{L/K} \circ \phi = \phi \circ \Tr_{L/K}$.
\end{Lemma}

\begin{proof}
If $L/K$ is not separable, then both sides are zero, so assume that $L/K$ is separable.
So $L^p \cap K = K^p$. 
Let $\ell_1$, \dots, $\ell_d$ be a basis for $L$ over $K$ and let $v_1$, $v_2$, \dots, $v_s$ be a basis for $K$ over $K^p$. 
So $\ell_i^p v_j$ is a basis for $L$ over $K^p$.
Since $\phi$ preserves $K$, we have $\phi(v_j) \in K$ for all $v_j$.

Let $x = \sum a_{ij}^p \ell_i^p v_j$, for some $a_{ij} \in K$. 
Then $\phi(x) = \sum a_{ij} \ell_i \phi(v_j)$ and, since  $\phi(v_j) \in K$,  we have $\Tr(\phi(x)) = \sum a_{ij} \phi(v_j) \Tr(\ell_i)$.
On the other hand, $\Tr(x) = \sum a^p_{ij} v_j \Tr(\ell_i^p) = \sum a^p_{ij} v_j \Tr(\ell_i)^p$ and $\phi(\Tr(x)) =   \sum a_{ij} \phi(v_j) \Tr(\ell_i)$.
\end{proof}

We also need some results about trace between normal rings:
\begin{Lemma}
Let $A$ and $B$ be integrally closed domains, with $K = \Frac A$ and $L = \Frac B$, and let $f: A \to B$ be an injection making $B$ into a finite $A$-module. Then $\Tr_{L/K}$ restricts to a map $B \to A$.
\end{Lemma}


\begin{proof}
If $L/K$ is inseparable, the trace map is $0$, so we may assume that $L/K$ is separable. Let $M$ be the Galois closure of $L$ over $K$, let $G = \mathrm{Gal}(M/K)$ and let $H \subset G$ be the stabilizer of $L$. Let $C$ be the integral closure of $A$ in $M$, so $B \subset C$. For any $\theta \in B$, we have $\Tr_{L/K} \theta = \sum_{g \in G/H} g \theta$, where the sum runs over a set of coset representatives for $G/H$. We have $\theta \in B \subset C$, and $C$ is taken to itself by $G$, so the sum is in $C \cap K$. Since $A$ is integrally closed, $C \cap K = A$.
\end{proof}

\begin{Lemma}  \label{Trace}
Let $A$, $B$ and $f$ be as in the previous lemma, and let $\fp$ be an ideal of $A$. If $\theta \in \sqrt{\fp B}$, then $\Tr_{L/K} \theta \in \fp$.
\end{Lemma}

\begin{proof}
We follow the same logic and notations as in the previous post. As before, we may assume that $L/K$ is separable. Note that $\theta \in \sqrt{\fp C}$ and $\sqrt{\fp C}$ is taken to itself by $G$, so $\Tr_{L/K} \theta = \sum_{g \in G/H} g \theta \in \sqrt{\fp C} \cap A$. Since $C$ is finite over $A$, this intersection is $\fp$.
\end{proof}

We now prove the central case of Theorem~\ref{MT}.

\begin{Proposition} \label{MainCase}
Let $\alpha: \cY \to \cX$ be a finite surjective map of irreducible varieties of characteristic $p$, with $\cX$ and $\cY$ both normal. Suppose that $p$ is greater than $\deg \alpha$.
Then, for $\phi$ any compatible splittings on $\cX(p)$ and $\cY(p)$, and $V$ any compatibly split subvariety of $\cX(p)$, the variety $\alpha^{-1}(V)^{\rr}$ is also split.
\end{Proposition}


\begin{proof}
We may assume that $V$ is irreducible, and we may localize to the generic point of $V$.
Let $A$ be the local ring of the generic point of $V$, and let $B$ be the semi-local ring of $Y$ at $f^{-1}(V)$. 
So $A$ and $B$ are noetherian integrally closed domains.
Let $K = \Frac A$ and $L = \Frac B$.

Let $\fp$ be the ideal of $V$ in $A$, set $J = \sqrt{\fp B}$ and let the prime decomposition of $J$ be $J = \bigcap \fq_i$.

Suppose, for the sake of contradiction, that $\phi(J) \not \subseteq J$, so there is some $x \in J$ with $\phi(x) \not \in J$ and thus $\phi(x) \not \in \fq_i$ for some $i$; without loss of generality, say $\phi(x) \not \in \fq_1$.
If we replace $x$ by $y^p x$ then we replace $\phi(x)$ by $y \phi(x)$; choosing $y$ appropriately, we can arrange that $\phi(x) \equiv 1 \bmod \fq_1$ and $\phi(x) \equiv 0 \bmod \fq_j$ for all $j \neq 1$.

By Lemma~\ref{Trace}, we have $\Tr_{L/K}(x) \in \fp$ so, since $\fp$ is compatibly split $\phi \left( Tr_{L/K}(x) \right) \in \fp$.
By Lemma~\ref{Key}, we deduce $\Tr_{L/K} \phi(x) \in \fp$. We will now compute $\Tr_{L/K} \phi(x) \bmod \fp$ directly and see that it is not zero, to obtain a contradiction.

Since $[L:K]=\deg \alpha < p$, the extension $L/K$ is separable.
Let $M$ be the Galois closure of $L$ over $K$; let $C$ be the integral closure of $B$ in $M$; let $G = \Gal(M/K)$; let $H$ be the stabilizer of $L$. 
So $\Tr_{L/K} \phi(x) = \sum_{g \in G/H} g \phi(x)$ where the sum runs over any collection of coset representatives for $G/H$. 

For $\fr$ a prime of $C$ lying above $\fp$, we have $\phi(x) \equiv 1 \bmod \fr$ if $\fr \cap B = \fq_1$ and $\phi(x) \equiv 0 \bmod \fr$ if $\fr \cap B$ is some other $\fq_j$.
Let $\fr_1$ be a prime of $C$ lying above $\fq_1$. 
So we have 
\[ \Tr_{L/K} \phi(x) =  \sum_{g \in G/H} g \phi(x) \equiv \# \{ g \in G/H : g \fr_1 \cap B = \fq_1 \} \bmod \fr_1. \]
The right hand side is a positive integer and at most $\#(G/H) = \deg  \alpha$. Since $p > \deg \alpha$, we deduce that $\Tr_{L/K} \phi(x) \not \equiv 0 \bmod \fr_1$. But $\fr_1 \cap A = \fp$, so we deduce that $\Tr_{L/K} \phi(x) \not \equiv 0 \bmod \fp$, as desired.


%
%
%
%
\end{proof}

We now prove Theorem~\ref{MT}. Our proof is by induction on $\dim \cX$.

We can immediately pass to irreducible components, and thus reduce to the case that $\cX$ and $\cY$ are integral.
Also, if $f$ is not surjective, then we can factor $f$ as $\cY \to f(\cY) \to \cX$. 
Any compatible splitting of $\cX$ and $\cY$ will pass to a splitting of $f(\cY)$, and we can reduce to the map $\cY\to f(\cY)$, whose image has smaller dimension, and apply induction.

Thus, we are reduced to the case that $\cX$ and $\cY$ are integral, and $f$ surjective.
Let $\widetilde{\cX}$ and $\widetilde{\cY}$ denote the normalizations of $\cX$ and $\cY$.
We write $\widetilde{X}(p)$ for $\widetilde{\cX} \times \Spec \FF_p$, and likewise for $\cY$.
Write $\widetilde{\alpha}$ for the map $\widetilde{\cY} \to \widetilde{\cX}$, and write $\mu$ and $\nu$ for the maps $\widetilde{\cX} \to \cX$ and $\widetilde{\cY} \to \cY$.

\[\xymatrix{
\widetilde{\cY} \ar[r]^{\widetilde{\alpha}} \ar[d]^{\nu} & \widetilde{\cX} \ar[d]^{\mu} \\
\cY \ar[r]^{\alpha} & \cX \\
}\]

\begin{Lemma}
For all but finitely many $p$, the fiber $\widetilde{\cX}(p)$ will be the normalization of $\cX(p)$. 
\end{Lemma}

\begin{proof} For all but finitely many primes $p$, the map $\widetilde{\cX}(p) \to \cX(p)$ will be birational, and it will be finite for all $p$, so it is enough to show $\widetilde{\cX}(p)$ is normal for all but finitely many $p$. We write $\cX(0)$ and $\widetilde{\cX}(0)$ for the generic fibers. 
 Since $\widetilde{\cX}$ is S2, we know that $\widetilde{\cX}(0)$ is S2. The set of $p \in \Spec \ZZ$ for which $\widetilde{\cX}(p)$ is S2 is constructible by \cite[EGA IV 9.9.3]{EGA}. A constructible subset of $\Spec \ZZ$ which contains the generic point is co-finite.
Since $\widetilde{\cX}$ is R1, we know that $\widetilde{\cX}(0)$ is R1 and, since $\QQ$ is of characteristic zero, $\widetilde{\cX}$ is also geometrically R1. The set of $p$ for which $\widetilde{\cX}(p)$ is geometrically R1 is constructible by \cite[EGA IV 9.9.5]{EGA} so, again, this set is co-finite.
\end{proof}

We will restrict ourselves to primes large enough that $\widetilde{\cX}(p)$ is the normalization of $\cX(p)$ and $\widetilde{\cY}(p)$ is the normalization of $\cY(p)$.
Any compatible splittings on $\cX(p)$ and $\cY(p)$ will give splittings of the normalizations  $\widetilde{\cX}(p)$ and $\widetilde{Y}(p)$, and all of these splittings will be compatible.

Let $\Spec A$ be an affine chart on $\cX$, and $\widetilde{A}$ the normalization of $A$.
Let $D$ be the conductor $\{ u \in A : u \widetilde{a} \in A \mbox{\ for\ all\ $\widetilde{a} \in \widetilde{A}$} \}$.
Similarly, let $D(p) = \{ u \in A/p A : u \widetilde{a} \in A/pA \mbox{\ for\ all\ } \widetilde{a} \in \widetilde{A}/p\widetilde{A}  \}$.
There is an obvious map $D/p D \to D(p)$ deriving from the map $A \to A/p A$.

\begin{Lemma}
For all but finitely many $p$, the map $D/p D \to D(p)$ is an isomorphism.
\end{Lemma}

\begin{proof}
Choose an $A$-spanning set $e_1$, $e_2$, \dots, $e_N$ for $\widetilde{A}$ as an $A$-module. 
Then $D$ is the kernel of the map $A \to (\widetilde{A}/A)^N$ given by $u \mapsto (u e_1, u_2, \ldots, u e_N)$.
Let $E$ be the image of this map $A \to (\widetilde{A}/A)^N$ and let $F$ be the cokernel.
So we have short exact sequences of $A$-modules $0 \to D \to A \to E \to 0$ and $0 \to E \to (\widetilde{A}/A)^N \to F \to 0$.

Since $E$ is a finitely generated $A$-module and $A$ is finitely generated over $\ZZ$, by Grothendieck's generic freeness theorem, $\Tor_1^A(E, A/pA)$ and $\Tor^1(F,A/pA)$ are $0$ for all but finitely many $p$.

Whenever this $\Tor$ vanishes, we have a short exact sequence $0 \to D/pD \to A/pA \to E/pE \to 0$ and an injection $E/p E \to (\widetilde{A}/p A)^N$, so $D/p D$ is the kernel of $A/p A \to (\widetilde{A}/p A)^N$. But $D(p)$ is also defined as this kernel, so $D(p) \cong D/pD$ and tracing through the diagrams shows that this isomorphism is the map described above.
\end{proof}

We will now further restrict our list of primes to those primes for which $D(p) = D/pD$. 

\begin{Lemma}
For any prime $p$, and any Frobenius splitting on $A/p$, the ideal $D(p)$ is compatibly split in $A/p A$.
\end{Lemma}

\begin{proof}
Let $u \in D(p)$ and $\widetilde{a} \in \widetilde{A}/p \widetilde{A}$. We must show that $\phi(u) \widetilde{a} \in A/p A$. 
We have $\phi(u) \widetilde{a} = \phi(u \widetilde{a}^p)$. 
Since $u \in D(p)$, we know that $u \widetilde{a}^p \in A/pA$, so $\phi(u \widetilde{a}^p) \in A/pA$ as desired.
\end{proof}
The construction of $D$ sheafifies; let $\Delta$  and $\widetilde{\Delta}$ be the corresponding subvarieties of $\cX$ and $\widetilde{\cX}$.
Then $\widetilde{\cX} \setminus \widetilde{\Delta} \to \cX \setminus \Delta$ is an isomorphism, and $\widetilde{\Delta} \to \Delta$ is a finite map.
By induction on dimension, there is an $N$ such that, for any $p>N$ and any choice of compatible splittings on $\Delta \times \FF_p$ and $\widetilde{\Delta} \times \FF_p$, all compatibly split subvarieties  of $\Delta \times \FF_p$ lift to $\widetilde{\Delta}  \times \FF_p$. 
From now on, we will choose $p$ larger than this $N$ (as well as obeying all of the other conditions on $p$).

Finally, we will restrict ourselves to $p > \deg \alpha$.

Take a $p$ large enough to obey all of our conditions, compatible splittings $\phi$ on $\cX(p)$ and $\cY(p)$, and a split subvariety $V$ of $\cX(p)$.
If the generic point of $V$ is not in $\Delta$, then $\mu^{-1}(V)$ is split because $\mu$ is an isomorphism away from $\Delta$.
If the generic point is in $\Delta$, then $\mu^{-1}(V)$ is split because $p$ we chose $p$ large enough that all compatibly split subvarieties of $\Delta$ lift to $\widetilde{\Delta}$.
So, either way, $\mu^{-1}(V)$ is split.

By Proposition~\ref{MainCase}, $\widetilde{\alpha}^{-1}(\mu^{-1}(V))$ is split. 
The image of a split variety is split, so $\nu (\widetilde{\alpha}^{-1}(\mu^{-1}(V))) = \alpha^{-1}(V)$ is split. 
This is the desired result. \qedsymbol

\end{document}